\documentclass{amsart}

\usepackage{amsmath}
\usepackage[dvips]{graphicx}
\usepackage{amsfonts}
\usepackage{amssymb}
\usepackage{amsthm}
\usepackage{mathrsfs}
\usepackage{amsrefs}
\usepackage[draft]{optional}
\usepackage{graphicx, color}
\usepackage{amssymb, paralist}
\usepackage{amsfonts, amsmath, wasysym}
\usepackage{pdfsync}
\usepackage{latexsym}
\usepackage{graphicx, color}
\usepackage{indentfirst}
\usepackage{bm}
\usepackage[active]{srcltx}

\newcommand{\R}{{\mathbb{R}}}

\allowdisplaybreaks[1]
\newtheorem{thm}{Theorem}

\newtheorem*{thmb}{Bray's Theorem}
\newtheorem*{thms}{Symmetrization Theorem}

\newtheorem{lemma}[thm]{Lemma}

\newtheorem*{conjecture}{Bray-Iga Conjecture}
\theoremstyle{remark}

\theoremstyle{remark}
\newtheorem{remark}[thm]{Remark}

\theoremstyle{definition}
\newtheorem{defn}[thm]{Definition}

\theoremstyle{remark}

\begin{document}

\title[A volumetric Penrose inequality]{A volumetric Penrose inequality for conformally flat manifolds}

\author{Fernando Schwartz}

\begin{abstract} We consider asymptotically flat Riemannian manifolds with nonnegative scalar curvature
that are conformal to $\R^{n}\setminus \Omega,~n\ge 3$, and so that their boundary is a minimal hypersurface.
(Here, $\Omega\subset \R^{n}$ is open bounded with smooth mean-convex boundary.)
We prove that the ADM mass of any such manifold is bounded below by 
$\left(V/\beta_{n}\right)^{(n-2)/n}$, where $V$ is the
Euclidean volume of $\Omega$ and $\beta_{n}$ is the volume of the Euclidean unit $n$-ball. This gives a partial proof to a conjecture of Bray and Iga \cite{brayiga}. Surprisingly, we do not require the boundary to be outermost.
\end{abstract}

\address{Department of Mathematics, University of Tennessee, Knoxville, USA}
\email{fernando@math.utk.edu}
\maketitle

\section{Introduction}

One of the major results in differential geometry is the positive mass inequality, which asserts that any asymptotically flat Riemannian manifold $M$ with nonnegative scalar curvature has nonnegative ADM mass. Furthermore, the inequality is rigid, in that
the ADM mass is strictly positive unless $M$ is isometric to the Euclidean space $\R^{n}$. This inequality was proved in 1979 by Schoen and Yau \cite{schoenyau}  for manifolds of dimension $n\le 7$ using minimal surface techniques. Witten 
\cite{witten} subsequently found a different argument based on spinors and the Dirac operator. (See also \cite{b} and \cite{pt}.) Witten's argument works for any spin manifold M, without any restrictions on the dimension.  

A refinement to the positive mass inequality in the case when black holes are present is the Riemannian Penrose inequality.  It
 asserts that
any asymptotically flat manifold $M$ with nonnegative scalar curvature containing an
outermost minimal hypersurface of area $A$ has ADM mass $m$ that satisfies
\begin{equation}\label{pen}\tag{RPI}
m\ge \frac{1}{2}\left(\frac{A}{\omega_{n-1}}\right)^{\frac{n-2}{n-1}},
\end{equation}
where $\omega_{n-1}$ is the area of the $(n-1)$-sphere $\mathbb{S}^{n-1}$.  This
inequality is also rigid, in that it is strict unless $M$ is isometric to the Riemannian Schwarzschild manifold.\footnote{See the beginning of \S3 
for the precise definition.}  This inequality was
first proved in three dimensions in 1997 by Huisken and Ilmanen  \cite{huiskenilmanen} for the case of a single black hole. In 1999, Bray \cite{bray} extended this result, still in dimension three, to the general case of multiple black holes using a different technique.   Later, Bray and Lee \cite{braylee} generalized Bray's proof for dimensions $n\le 7$, with the extra requirement that $M$ be spin
for the rigidity statement.\\

A special situation arises if we restrict ourselves to the case of {\it conformally flat} manifolds.  There,  the proof of the positive mass theorem follows from Green's formula.  In view of this, Bray and Iga   conjectured the following 
 in \cite{brayiga}.
 
 \begin{conjecture}
 The Riemannian Penrose inequality holds for arbitrary-dimensional conformally flat, asymptotically
 flat manifolds with nonnegative scalar curvature.
\end{conjecture}

  Bray and Iga gave a partial proof of this conjecture
  in \cite{brayiga}.  They showed that the RPI holds
   with suboptimal constant $c<1$ on manifolds
  conformal to the flat metric on $\R^3$ minus the origin.  Specifically, they
  proved that $m\ge c\sqrt{A/16\pi}$ where $A$ is the infimum of the areas
  of all surfaces enclosing the origin.\\


In this paper we prove what we call a ``volumetric" Penrose inequality for conformally flat manifolds which works in arbitrary dimensions, thus partially answering the Bray-Iga conjecture.
An important point to note concerning our inequality is that it is also a
suboptimal inequality, in that it is weaker than than the RPI in the 
cases where the latter is applicable.
 On the other hand, our Theorem works in all dimensions $n\ge 3$. This is
particularly interesting in dimensions 8 and above, where no such results
(asides from spherically symmetric manifolds) were known to exist.\footnote{At the time of submission the author found out that
Lam, a student of Bray, has proved the positive mass inequality for graphs 
of asymptotically flat functions over $\R^n$, and the Riemannian Penrose inequality for graphs on $\R^n$ with convex boundaries, all this for $n\ge 3$.}

The precise statement of our main theorem
is the  following.

\begin{thm} \label{prs}
Suppose that $(M^{n},g), ~n\ge 3,$
is an asymptotically flat $n$-dimensional manifold  with nonnegative scalar curvature which is isometric to  $(\R^{n}\setminus \Omega,
u^{4/(n-2)}\delta_{ij})$, where $\Omega\subset \R^{n}$ is an open bounded set with smooth mean-convex boundary (i.e. having positive mean curvature), and 
$u$ is normalized so that $u\to 1$ towards infinity. 
If the boundary of $M$ is a minimal hypersurface, then
\begin{equation}\label{penrose}
m\ge \left(\frac{V}{\beta_n}\right)^{\frac{n-2}{n}},
\end{equation}
where $m$ is the ADM mass of $(M,g)$, $V$ is the volume of $\Omega$ with respect to the Euclidean metric, and
$\beta_{n}$ is the volume of the Euclidean unit $n$-ball.  
\end{thm}

The requirements that $(M,g)$ be conformally flat and that the boundary of 
$\Omega$ have positive mean curvature appears to be quite stringent,
 but not so much from a topological 
point of view.  For example,
the manifolds we constructed in 
\cite{schwartz08}, which are the only known asymptotically flat manifolds with nonnegative scalar curvature having
outermost minimal hypersurfaces which are not topological spheres,\footnote{The outermost minimal hypersurfaces are, topologically,
a product of spheres.} are all conformally flat, and their respective
$\Omega$'s have mean-convex boundary.  Actually, since we do not require the 
boundary of $M$ to be an {\it outermost} minimal hypersurface, there are many topologically-inequivalent examples of manifolds which satisfy the hypotheses of our theorem.
Indeed, from the construction of \cite{schwartz08} it follows fairly easily that one can find examples of scalar flat, asymptotically flat
 manifolds having minimal
boundary which is, topologically, the boundary of any given handlebody in $\R^{n}$. Using appropriate scalings these can be made mean-convex as well.  

\begin{remark}
For the special case of a Schwarzschild metric, it can be easily checked that
the RHS of inequality \eqref{penrose} is $1/2$ the RHS of equation \eqref{pen}.
This is, the volumetric Penrose inequality is off by a factor of 2 
 from being optimal.
\end{remark}

\begin{remark} \label{rem}
Our theorem does not require that the boundary of $M$ be outermost, in
contrast with the standard RPI\footnote{This assumption is necessary in the RPI, for it is well known that 
counterexamples  may be obtained by taking spherically 
symmetric metrics with fixed mass and arbitrarily large
minimal, but not outermost, boundary, like in p. 358 of \cite{huiskenilmanen}.}.
 This may seem odd at first.  Nevertheless, since a non-outermost minimal hypersurface bounds a domain that is contained in the domain that 
 the outermost minimal hypersurface bounds,  inequality \eqref{penrose} only gives a weaker bound 
 when applied to a
 non-outermost  minimal boundary compared to it being applied to the external region of the outermost minimal hypersurface.  Also, notice that we need to impose $u\to 1$ towards infinity to
 get rid of would-be counterexamples where the volume of $\Omega$ can be made arbitrarily 
 large maintaining the mass bounded.\\
\end{remark}

\noindent
{\it Outline of the proof.}   We  use a straightforward extension of a theorem of
Bray to spin manifolds and obtain a lower bound for the ADM mass 
of $(M,g)$ in terms of
the capacity of its boundary.
We then focus on finding an estimate for the capacity
of the boundary.   It turns out that, in the conformally flat case,
this can be done  using a spherical symmetrization 
trick, so long as we can find appropriate bounds for the conformal factor.
In order to obtain these we requite that the boundary of $\Omega$ be mean-convex. 
 \\

 We should
mention that Bray and Miao also exploit the relationship between mass and capacity in   \cite{braymiao}, but their estimates
go in the opposite direction.  In their beautiful
work  they find upper bounds for the capacity of surfaces in terms of the Hawking mass, all this
inside asymptotically flat three dimensional manifolds with nonnegative scalar curvature.  
The proof of their main theorem relies on  the  monotonicity of the Hawking mass along
inverse mean curvature flow; this is known to work only in dimension three.  Their result was 
inspired by an earlier result of Bray and Neves   
\cite{brayneves}, where similar techniques were used for computing Yamabe invariants.\\

{\bf Acknowledgments.}  This work was mostly carried out while visiting the IMPA in Rio de Janeiro, Brazil.
I  thank the University of Tennessee's Professional Development Award for providing with partial support for the trip.  
I thank the IMPA for their hospitality, and 
 Fernando Cod\'a Marques for some useful conversations. 
 I thank Hugh Bray for his useful comments after carefully proofreading 
  a first draft, and Jeff Jauregui for pointing out a redundant 
  argument in the proof of Lemma 11.

\section{Preliminaries}

We begin by recalling some classical facts about spherical symmetrization in $\R^{n}$.

\begin{defn}\label{symm}
Let $u$ be a function in $W^{1,p}(\R^{n})$.  Its {\it spherical
symmetrization}, $u^{*}(x)\equiv u^{*}(|x|),$ is the unique radially symmetric function on $\R^{n}$ which is decreasing on $|x|$, and so that 
the Lebesgue measure of the super-level sets of $u^{*}$ equals the Lebesgue measure of the super-level sets of $u$.  More precisely, $u^*$ is defined as the unique decreasing spherically symmetric function on $\R^n$ so that 
$\mu \{u\ge K\}=\mu \{u^{*}\ge K\}$ for all $K\in\R$.  
 \end{defn}

The following result is a classical theorem in analysis which can be traced back to a principle used by P\'olya and Szeg\"o \cite{polyaszego}. (See also \cite{talenti}, \cite{hilden}.)

\begin{thms}[\cite{polyaszego}] Spherical symmetrization preserves $L^{p}$ norms and decreases
$W^{1,p}$ norms. 
\end{thms}

We need the above result for a calculation inside the proof of the main theorem.  We now introduce the notions of asymptotical flatness,
ADM mass, and capacity, and give a result of Bray  concerning these quantities.

\begin{defn}
Let $n \ge 3$.  A Riemannian manifold $(M^{n}, g)$ is said to be {\it asymptotically flat} if there is a compact set 
$K \subset M$ such that $M\setminus K$ is diffeomorphic to $\R^{n} \setminus B_{1}(0)$, and in this coordinate chart the metric $g_{ij}$ satisfies
$$g_{ij} =\delta_{ij} +O(|x|^{-p}), ~g_{ij,k} = O(|x|^{-p-1}), ~g_{ij,kl} = O(|x|^{-p-2}),~
R_{g} =O(|x|^{-q}),$$
for some $p > (n - 2)/2$ and some $q > n$, where the commas denote partial derivatives in the coordinate chart, and $R_{g}$ is the scalar curvature of $g$.
\end{defn}

For an asymptotically flat manifold $(M,g)$, it is well known that the limit 
\begin{equation*}
m(g)=\lim_{r\to\infty} \frac{1}{2(n-1)\omega_{n-1}} \int_{S_{r}}(g_{ij,i} -g_{ii,j})\nu_{j} dA
\end{equation*}
exists, where $\omega_{n-1}$ is the area of the standard 
unit $(n-1)$-sphere, $S_{r}$ is the coordinate sphere of radius $r$,  $\nu$ is its outward unit normal, and $dA$ is the Euclidean area element on $S_{r}$ . 

\begin{defn} The quantity $m=m(g)$ from above is called the {\it ADM mass} of $(M^{n},g)$.
\end{defn}

\noindent
This notion of mass was first considered by Arnowitt, Deser, and Misner in \cite{adm}. Later,
Bartnik showed that the ADM mass is a Riemannian invariant, independent of choice of asymptotically flat coordinates, cf. Section 4 of \cite{b}. (See also
\cite{piotr}.)

\begin{defn} The {\it capacity} of the boundary $\Sigma$ of 
a complete, asymptotically flat manifold $(M^{n},g)$ is
\begin{equation*}
\mathcal{C}(\Sigma,g)=\inf\left\{ \frac{1}{(n-2)\omega_{n-1}}\int_{M}|\nabla\varphi|^{2} dV\right\},
\end{equation*}
	
where the infimum is taken over all smooth $0\le\varphi(x)\le 1$ which go to zero at infinity and equal to one on the boundary $\Sigma$. 
\end{defn}

\begin{remark}
The above definition of capacity differs slightly from the standard definition of capacity.  We ask that the functions
considered in the infimum satisfy the extra hypothesis
$0\le\varphi(x)\le 1$, which is required  for the  proof of Lemma 
\ref{eas}.
Nevertheless, this extra assumption does not affect the outcome of the infimum, since with or without it the infimum is attained by a positive harmonic function no greater than one. (Cf. equation (86) of \cite{bray}.)
\end{remark}

The following theorem of Bray is central to our purposes since it establishes a 
relationship between mass and capacity.

\begin{thmb}[\cite{bray}]  Let $(M^{n},g),~n\ge 3$ be an asymptotically flat manifold with boundary so that either the double of $M$ is spin,  or $M$
has dimension less than 8.  Assume further that $M$
has nonnegative scalar curvature and minimal boundary $\Sigma$.  Let $m$ be its ADM mass. Then
\begin{equation*}
m \ge \mathcal{C}(\Sigma,g),
\end{equation*}
with equality if and only if $(M^{n},g)$ is a Riemannian Schwarzschild manifold outside its outermost minimal hypersurface $\Sigma$. 
\end{thmb}

\begin{remark}  Bray's original version of the above theorem, which is Theorem 9 of \cite{bray}, 
does not include the case of the double of
$M$  being spin, but for our purposes this is a natural assumption.  It is easy to see that a slight modification of Bray's proof  using Witten's positive mass theorem   whenever necessary
gives a  proof of the statement above.
\end{remark}


Finally, we cite a quick fact about spin geometry that we will use in the
proof of the main theorem.  (Cf. p.90 of \cite{ML}.)

\begin{lemma}\label{BW}
Let $M$ be diffeomorphic to $\R^n\setminus \Omega$, where $\Omega$ is
a bounded open subset of $\R^n$ with smooth boundary. Then both $M$ and its double 
across the boundary are spin.
\end{lemma}

\section{Proof of Theorem \ref{prs}}\label{proof}

\noindent
Throughout this section we will be using two different metrics:

\begin{itemize}
\item[(i)] the Euclidean metric $\delta_{ij}$ on $\R^n$,
\item[(ii)] the conformally flay metric of $(M,g)$ given by $g=u^{4/(n-2)}\delta_{ij}$, where $u>0$ is a smooth function defined on $\R^n\setminus \Omega$,
\end{itemize}

\noindent
Standard quantities depending on the metric, like covariant derivatives, volume forms, norms and so on,
will be denoted by, respectively,
\begin{itemize}
\item[(i)] $\nabla_{0},~dV_{0},~|\cdot|_{0}$, 
\item[(ii)]$\nabla_{g},~dV_{g},~|\cdot|_{g}$.
\end{itemize}

We begin by proving an estimate
for the conformal factor $u$, which is of independent interest.  (Here is where we need that the boundary of 
$\Omega$ be mean-convex.)

\begin{lemma}\label{one}
Suppose that $(M^{n},g)$
is an asymptotically flat $n$-dimensional manifold  with nonnegative scalar curvature which is isometric to  $(\R^{n}\setminus \Omega,
u^{4/(n-2)}\delta_{ij})$, where $\emptyset\neq\Omega\subset \R^{n}$ is an open bounded set with smooth mean-convex boundary.  Assume that
the boundary of $M$ is minimal,  and that 
$u$ is normalized so that $u\to 1$ towards infinity.  Then $u\ge 1$ on $M$.  
\end{lemma}

\begin{proof} 
Recall that the transformation law for the scalar curvature under conformal changes of the metric is given by $R_g=\frac{4(n-1)}{n-2}u^{-(n+2)/(n-2)}
 (-\Delta_0+\frac{n-2}{4(n-1)}R_0)u$, where $\Delta_0$ is the Euclidean 
 Laplacian and $R_0$ is the Euclidean scalar curvature, namely $R_0\equiv 0$.
 Since we assume that $R_g\ge 0$, it follows that $u$ is superharmonic
 on $M$. Therefore, $u$ achieves its minimum value at either infinity or at 
 the boundary $\partial \Omega$.  At infinity $u$ goes to one.  We now show that
 at the boundary it does not achieve its minimum, and so it must be everywhere
 greater or equal than one.\\
  {\it Claim.} $u$ does not achieve its minimum on the boundary  $\partial \Omega$.\\
 From hypothesis, the boundary of $M$ is a minimal hypersurface.  This is,
 the mean curvature of the boundary of $M$ is zero with respect to 
 the  metric $g=u^{4/(n-2)}\delta_{ij}$.  Now, the transformation law
 for the mean curvature under the conformal change of the metric 
  $g=u^{4/(n-2)}\delta_{ij}$ is given by
 $h_g=\frac{2}{n-2}u^{-n/(n-2)}(\partial_\nu+\frac{(n-2)}{2}h_0)u$, where
 $h_0$ is the Euclidean mean curvature and $\nu$ is the outward-pointing
 normal.  Since we have assumed that the boundary of $\Omega$ is mean
 convex, i.e. that $h_0>0$, it follows that $\partial_\nu u<0$ on all of the boundary
 of $\Omega$.  This way, $u$ decreases when we move away from the boundary towards the interior of $M$.  From this it follows
 that $u$ cannot achieve its minimum on the boundary.  This proves
 the claim, and the Lemma follows. 
\end{proof}

We now bring spherical symmetrization into the picture.
Suppose that $0\le\varphi\le 1$ is a smooth function on $\R^{n}\setminus \Omega$ 
which is exactly 1 on the boundary $\partial \Omega$ and converges to 0 at infinity.  
We may extend this function to a function $\tilde \varphi$ defined on all of $\R^{n}$, given by
\begin{equation*}
\tilde \varphi =\left\{
\begin{tabular}{ll}
1 & in  $\Omega$,\\
$\varphi$ & outside  $\Omega$.
\end{tabular}
\right.
\end{equation*}
 (Notice that $\tilde \varphi$  is Lipschitz.)
 Now consider $(\tilde \varphi)^{*}$,  the spherical symmetrization of $\tilde\varphi$, which is defined on all of $\R^{n}$.
(See Definition \ref{symm}.)

\begin{defn}
Let $\varphi$ be as above, and let $V$ be the Euclidean volume of $\Omega$.  We define $\varphi^{*}$
to be the restriction of $(\tilde \varphi)^{*}$ to $\R^{n}\setminus B_R(0)$, where $R=R(V)$ is the 
radius of the Euclidean ball of volume $V$, namely $R=(V/\beta_n)^{1/n}$.
\end{defn}

\begin{lemma}\label{eas}
Let $\varphi$ be as  above.  Then 
\begin{equation*}
\int_{M}
|\nabla_{0}\varphi|_{0}^{2}dV_{0}\ge\int_{\R^{n} \setminus B_R(0) }|\nabla_{0}\varphi^{*} |_{0}^{2}dV_{0}.
\end{equation*}
\end{lemma}

\begin{proof} Recall that $M=\R^{n}\setminus \Omega$.
Since $\tilde\varphi$ is Lipschitz and is constant inside $\Omega$, it follows that 
$\int_{\R^{n}\setminus \Omega}|\nabla_{0}\varphi|_{0}^{2}dV_{0}=\int_{\R^{n}}|\nabla_{0}\tilde\varphi|_{0}^{2}dV_{0}$.  From the Symmetrization Theorem applied to $\tilde\varphi$, we obtain that 
$\int_{\R^{n}}|\nabla_{0}\tilde\varphi|_{0}^{2}dV_{0}\ge \int_{\R^{n}}|\nabla_{0}(\tilde\varphi)^{*}|_{0}^{2}dV_{0}$.
But since $0\le\tilde\varphi\le 1$ is constant and equal to one on $\Omega$,  it follows that 
$(\tilde\varphi)^{*}$ is also constant and equal to one on the ball $B_R(0) $, where $R=(V/\beta_n)^{1/n}$
and $V$ is the Euclidean volume of $\Omega$.  This way, $\int_{\R^{n}}|\nabla_{0}(\tilde\varphi)^{*}|_{0}^{2}dV_{0}= \int_{\R^{n}\setminus B_R(0)}|\nabla_{0}\varphi^{*}|_{0}^{2}dV_{0}$.  Putting this inequalities together
gives a proof of the lemma.
\end{proof}

\begin{lemma} \label{easy}
Let $g=u^{4/(n-2)}\delta_{ij}$.  We have that
 $|\nabla_{g}\varphi|_{g}^{2}=u^{-4/(n-2)}|\nabla_{0}\varphi|_{0}^{2}$,
and $dV_{g}=u^{2n/(n-2)}dV_{0}$.
\end{lemma}

\begin{proof} Straightforward calculation.
\end{proof}

We  now   prove the key lemma.

\begin{lemma}\label{basic}  Let $(M,g)$ be as in Theorem 1 and
consider a smooth function $0\le\varphi\le 1$ on $M$ so that $\varphi=1$ on 
$\partial M$ and $\varphi\to 0$ towards infinity. Then
\begin{equation*}
\int_{M}|\nabla_{g}\varphi|^{2}_{g}dV_{g}\ge  \int_{\R^{n}\setminus B_R(0)}|\nabla_{0}\varphi^{*}|_{0}^{2}dV_{0},
\end{equation*}
where  $R=(V/\beta_n)^{1/n}$.
\end{lemma}

\begin{proof} Using Lemma \ref{easy} we obtain
\begin{align*}
\int_{M}|\nabla_{g}\varphi|_{g}^{2}dV_{g} = &
     \int_{M}u^{-4/(n-2)}|\nabla_{0}\varphi|_{0}^{2}u^{2n/(n-2)}dV_{0} =
 \int_{M}u^{2}|\nabla_{0}\varphi|_{0}^{2}dV_{0} \\
    \ge & (\inf_{M} u^{2})\int_{M}|\nabla_{0}\varphi|_{0}^{2}dV_{0}, \\
    \intertext{but $u\ge 1$ by Lemma \ref{one}; this together with Lemma \ref{eas} gives}
   \ge & \int_{M}|\nabla_{0}\varphi|_{0}^{2}dV_{0} \ge  \int_{\R^{n}\setminus B_R}|\nabla_{0}\varphi^{*}|_{0}^{2}dV_{0}.
 \end{align*}

\end{proof}

\begin{proof}[Proof of Theorem \ref{prs}] The double of $M$  is spin 
from Lemma \ref{BW}.  Thus, we may apply 
Bray's Theorem and obtain that $m(g)\ge \mathcal{C}(\Sigma,g)$.  From 
Lemma \ref{basic} it follows that $\mathcal{C}(\Sigma,g)\ge \mathcal{C}(S_{R},\delta_{ij})$
where $S_{R}$ is the boundary of $\R^{n}\setminus B_R(0)$.  
This last quantity is easily computed and known to be $\left(\frac{V}{\beta_{n}}\right)^{\frac{n-2}{n}}$.
\end{proof}


\begin{bibdiv}
\begin{biblist}
 
 \bib{adm}{article}{
   author={Arnowitt, R.},
   author={Deser, S.},
   author={Misner, C. W.},
   title={Coordinate invariance and energy expressions in general
   relativity. },
   journal={Phys. Rev. (2)},
   volume={122},
   date={1961},
   pages={997--1006}
}
 
 \bib{b}{article}{
   author={Bartnik, R.},
   title={The mass of an asymptotically flat manifold},
   journal={Comm. Pure Appl. Math.},
   volume={39},
   date={1986},
   number={5},
   pages={661--693},
   issn={0010-3640},

}
 

 
 \bib{bray}{article}{
    author={Bray, H. L.},
     title={Proof of the Riemannian Penrose inequality using the positive
            mass theorem},
   journal={J. Differential Geom.},
    volume={59},
      date={2001},
    number={2},
     pages={177\ndash 267},
      issn={0022-040X},
}

\bib{brayiga}{article}{
   author={Bray, H. L.},
   author={Iga, K.},
   title={Superharmonic functions in $\bold R^n$ and the Penrose
   inequality in general relativity},
   journal={Comm. Anal. Geom.},
   volume={10},
   date={2002},
   number={5},
   pages={999--1016},
   issn={1019-8385},
 
}

 \bib{braylee}{article}{
   author={Bray, H. L.},
   author={Lee, D. A.},
   title={On the Riemannian Penrose inequality in dimensions less than
   eight},
   journal={Duke Math. J.},
   volume={148},
   date={2009},
   number={1},
   pages={81--106},
   issn={0012-7094},

}


\bib{braymiao}{article}{
   author={Bray, H. L.},
   author={Miao, P.},
   title={On the capacity of surfaces in manifolds with nonnegative scalar
   curvature},
   journal={Invent. Math.},
   volume={172},
   date={2008},
   number={3},
   pages={459--475},
   issn={0020-9910},

}
 
 \bib{brayneves}{article}{
   author={Bray, H. L.},
   author={Neves, A.},
   title={Classification of prime 3-manifolds with Yamabe
   invariant greater than $\Bbb{RP}^3$},
   journal={Ann. of Math. (2)},
   volume={159},
   date={2004},
   number={2},
   pages={407--424},

}
 
 \bib{piotr}{article}{
   author={Chru{\'s}ciel, P.},
   title={Boundary conditions at spatial infinity from a Hamiltonian point
   of view},
   conference={
      title={Topological properties and global structure of space-time
      (Erice, 1985)},
   },
   book={
      series={NATO Adv. Sci. Inst. Ser. B Phys.},
      volume={138},
      publisher={Plenum},
      place={New York},
   },
   date={1986},
   pages={49--59},
   }
 
 \bib{hilden}{article}{
   author={Hild{\'e}n, K.},
   title={Symmetrization of functions in Sobolev spaces and the
   isoperimetric inequality},
   journal={Manuscripta Math.},
   volume={18},
   date={1976},
   number={3},
   pages={215--235},
   issn={0025-2611},
}

\bib{huiskenilmanen}{article}{
   author={Huisken, G.},
   author={Ilmanen, T.},
   title={The inverse mean curvature flow and the Riemannian Penrose
   inequality},
   journal={J. Differential Geom.},
   volume={59},
   date={2001},
   number={3},
   pages={353--437},
   issn={0022-040X},
}

\bib{ML}{book}{
   author={Lawson, H. B., Jr.},
   author={Michelsohn, M.-L.},
   title={Spin geometry},
   series={Princeton Mathematical Series},
   volume={38},
   publisher={Princeton University Press},
   place={Princeton, NJ},
   date={1989},
   pages={xii+427},
   isbn={0-691-08542-0},
}

\bib{pt}{article}{
   author={Parker, T.},
   author={Taubes, C. H.},
   title={On Witten's proof of the positive energy theorem},
   journal={Comm. Math. Phys.},
   volume={84},
   date={1982},
   number={2},
   pages={223--238},
   issn={0010-3616},
}

\bib{polyaszego}{book}{
   author={P{\'o}lya, G.},
   author={Szeg{\"o}, G.},
   title={Isoperimetric Inequalities in Mathematical Physics},
   series={Annals of Mathematics Studies, no. 27},
   publisher={Princeton University Press},
   place={Princeton, N. J.},
   date={1951},
   pages={xvi+279},
}

\bib{schoenyau}{article}{
   author={Schoen, R.},
   author={Yau, S.-T.},
   title={On the proof of the positive mass conjecture in general
   relativity},
   journal={Comm. Math. Phys.},
   volume={65},
   date={1979},
   number={1},
   pages={45--76},
   issn={0010-3616},
}

\bib{schwartz08}{article}{
   author={Schwartz, F.},
   title={Existence of outermost apparent horizons with product of spheres
   topology},
   journal={Comm. Anal. Geom.},
   volume={16},
   date={2008},
   number={4},
   pages={799--817},
   issn={1019-8385},
}

\bib{talenti}{article}{
   author={Talenti, G.},
   title={Best constant in Sobolev inequality},
   journal={Ann. Mat. Pura Appl. (4)},
   volume={110},
   date={1976},
   pages={353--372},
   issn={0003-4622},
}

\bib{witten}{article}{
   author={Witten, E.},
   title={A new proof of the positive energy theorem},
   journal={Comm. Math. Phys.},
   volume={80},
   date={1981},
   number={3},
   pages={381--402},
   issn={0010-3616},
}


\end{biblist}
\end{bibdiv}
\end{document}